\documentclass[]{amsart}   
\usepackage{amsmath}
\usepackage[mathscr]{eucal} 
\usepackage{amssymb}
\usepackage{latexsym}
\usepackage{amsthm} 
\theoremstyle{plain}
\newtheorem{theorem}{Theorem}[section]
\newtheorem{lemma}[theorem]{Lemma}  

\newtheorem{theorema}{Theorem A\!\!}

\newtheorem{corollary}[theorem]{Corollary}

\newcommand{\supp}{\mathop{\mathrm{supp}}\nolimits}

\theoremstyle{definition}

\numberwithin{equation}{section} 
\newtheorem{remark}[theorem]{Remark}

\theoremstyle{remark}

\def\Xint#1{\mathchoice
{\XXint\displaystyle\textstyle{#1}}%
{\XXint\textstyle\scriptstyle{#1}}%
{\XXint\scriptstyle\scriptscriptstyle{#1}}%
{\XXint\scriptscriptstyle\scriptscriptstyle{#1}}%
\!\int}
\def\XXint#1#2#3{{\setbox0=\hbox{$#1{#2#3}{\int}$}
\vcenter{\hbox{$#2#3$}}\kern-.5\wd0}}

\def\dashint{\Xint-}

\title[ Hardy spaces on homogeneous groups]
{ Hardy spaces on homogeneous groups and 
Littlewood-Paley functions}  
\author{Shuichi Sato} 
 
\begin{document} 

\address{Department of Mathematics,
Faculty of Education, Kanazawa University, Kanazawa 920-1192, Japan}
\email{shuichi@kenroku.kanazawa-u.ac.jp}
\begin{abstract} 
We establish a characterization of the Hardy spaces on the homogeneous groups 
in terms of the Littlewood-Paley functions. The proof is based on 
vector-valued inequalities shown by applying the Peetre maximal function.  
\end{abstract}
  \thanks{2010 {\it Mathematics Subject Classification.\/}
  Primary  42B25; Secondary 42B30.
  \endgraf
  {\it Key Words and Phrases.} Hardy spaces, Heisenberg group, 
   vector-valued inequalities, 
  Littlewood-Paley functions.  }
\thanks{The author is partly supported
by Grant-in-Aid for Scientific Research (C) No. 16K05195,
Japan Society for the  Promotion of Science.}

\maketitle 

\section{Introduction}  

Let $\Bbb R^n$ be the $n$ dimensional Euclidean space. 
In this note we assume that $n\geq 2$.     
We also consider a structure on $\Bbb R^n$ which makes $\Bbb R^n$   
 a homogeneous group $\Bbb H$ equipped with 
multiplication given by a polynomial mapping.  
This requires that we have a dilation family 
$\{A_t\}_{t>0}$ on $\Bbb R^n$ of the form 
 $$A_tx=(t^{a_1}x_1,t^{a_2}x_2,\dots, t^{a_n}x_n), \quad 
 x=(x_1,\dots, x_n), $$ 
with some real numbers $a_1, \dots , a_n$ satisfying 
 $1=a_1\leq a_2\leq \dots \leq a_n$ 
such that each $A_t$ is an automorphism of 
the group structure    
 (see \cite{FoS},  \cite{T} and \cite[Section 2 of Chapter 1]{NS}).     
More precisely, in addition to the Euclidean structure, 
$\Bbb H$ is equipped with 
a homogeneous nilpotent Lie group structure and we have the following: 
\begin{enumerate}  
\renewcommand{\labelenumi}{(\arabic{enumi})}  
\item Lebesgue measure is a bi-invariant Haar measure; 
\item the identity is the origin $0$ and $x^{-1}=-x$;  
\item $(\alpha x)(\beta x)= \alpha x+\beta x$ for  
$x\in \Bbb H$, $\alpha, \beta \in \Bbb R$; 
\item $A_t(xy)=(A_tx)(A_ty)$ for  $x, y \in \Bbb H$, $t>0$;  
\item if $z=xy$, then $z_k=P_k(x,y)$, where $P_1(x,y)=x_1+y_1$ and 
$P_k(x,y)=x_k+y_k+R_k(x,y)$ for $k\geq 2$ with  a polynomial   
$R_k(x,y)$ depending only on 
$x_1, \dots ,x_{k-1}, y_1, \dots ,y_{k-1}$, which can be written as 
$$R_k(x,y)=\sum_{|I|\neq 0, |J|\neq 0, a(I)+a(J)=a_k}c_{I,J}^{(k)}
x^Iy^J. $$   
(See Sections $2.1$ and $2.2$ below for the notation.) 
 \end{enumerate}  
 Let $|x|$ be the Euclidean norm  for $x\in \Bbb R^n$. We    
  have a norm function $\rho(x)$ which is homogeneous of degree one 
with respect to the dilation $A_t$; by this we 
mean that $\rho(A_tx)=t\rho(x)$ 
for  $t>0$ and $x\in \Bbb H$. We may assume the following:  
\begin{enumerate}  
\renewcommand{\labelenumi}{(\arabic{enumi})}  
\item[(6)]  $\rho$ is continuous on $\Bbb  R^n$ and smooth in 
 $\Bbb H\setminus \{0\}$;  
\item[(7)]  $\rho(x+y)\leq \rho(x)+\rho(y)$ and 
 $\rho(xy)\leq c_0(\rho(x)+\rho(y))$ for some 
 constant $c_0\geq 1$ and $\rho(x^{-1})=\rho(x)$;  
\item[(8)]  
we have 
\begin{gather*} 
c_1|x|^{\alpha_1}\leq \rho(x)\leq c_2|x|^{\alpha_2} \quad 
\text{if $\rho(x)\geq 1$}, 
\\ 
c_3|x|^{\beta_1}\leq \rho(x)\leq c_4|x|^{\beta_2} 
\quad \text{if $\rho(x)\leq  1$},   
\end{gather*}  
with some positive constants $c_j$, $\alpha_k$, $ \beta_k$, 
 $1\leq j\leq 4$, $1\leq k\leq 2$;  
\item[(9)] $\rho(x)\leq 1$ if and only if $|x|\leq 1$ and 
 the unit sphere $\Sigma=\{x\in \Bbb H: \rho(x)=1\}$ with respect 
to $\rho$ coincides with  $S^{n-1}=\{x\in \Bbb R^n: |x|=1\}$.    
 \end{enumerate}    
\par 
The polar coordinate expression of the Lebesgue measure 
$dx=t^{\gamma-1}\ dS\,dt$ is useful, where $\gamma=a_1+\dots +a_n$ 
(the homogeneous dimension).  By this we mean that  
$$\int_{\Bbb H}f(x)\,dx=\int_0^\infty\int_\Sigma f(A_t\theta)t^{\gamma-1}\,
dS(\theta)\,dt $$ 
with $dS=\omega\, dS_0$ for  appropriate functions $f$,
  where $\omega$ is a strictly positive $C^\infty$ function on $\Sigma$ and 
$dS_0$ denotes the Lebesgue surface measure on $\Sigma$.    
\par 
We recall the Heisenberg group $\Bbb H_1$ as 
an example of a homogeneous group. 
Let us define the multiplication 
$$(x_1,x_2,x_3)(y_1,y_2,y_3)=(x_1+y_1,x_2+y_2,x_3+y_3+(x_1y_2-x_2y_1)/2), $$  
 $(x_1,x_2,x_3), (y_1,y_2,y_3)\in \Bbb R^3$.  Then this
  group law defines the  Heisenberg group $\Bbb H_1$ with the 
  underlying manifold $\Bbb R^3$, where the dilation 
 $A_t(x_1,x_2,x_3)=(tx_1,tx_2,t^2x_3)$ is an automorphism. 
 \par    
We consider the Littlewood-Paley $g$ function on $\Bbb H$ defined by 
\begin{equation}\label{lpop}
g_{\varphi}(f)(x) = \left( \int_0^{\infty}|f*\varphi_t(x)|^2
\,\frac{dt}{t} \right)^{1/2},    
\end{equation} 
where $f\in \mathscr S'$, $\varphi \in \mathscr S$ satisfying 
$\int_{\Bbb H}\varphi\,dx=0$ and  
$\varphi_t(x)=t^{-\gamma}\varphi(A_t^{-1}x)$. 
Here $\mathscr S'$ denotes the space of tempered distributions and 
$\mathscr S$ the Schwartz space, which are the same as those 
in the Euclidean case (see \cite{SW});
also the convolution $F*G$ for $F, G\in L^1$ is defined by 
$$ F*G(x)=\int_{\Bbb H}F(xy^{-1})G(y)\, dy=
\int_{\Bbb H}F(y)G(y^{-1}x)\, dy. $$  
We refer to \cite{DS} and \cite{T, S4,S5} for the study of 
Littlewood-Paley operators and  singular integrals, respectively, 
on $L^p$ spaces for homogeneous groups,  $1\leq p<\infty$. 
\par 
In this note we prove a characterization of Hardy spaces $H^p$, $0<p\leq 1$, 
on  $\Bbb H$ (see Section 2.3 below) in terms of the Littlewood-Paley $g$ 
functions.    
We first recall related results in the Euclidean case.  
Let $\varphi^{(\ell)}$, $\ell=1, 2, \dots, M$, be functions in 
$\mathscr S(\Bbb R^n)$  
satisfying the non-degeneracy condition 
\begin{equation}\label{nondegeneracy} 
\inf_{\xi\in \Bbb R^n\setminus\{0\}}\sup_{t>0}\sum_{\ell=1}^M
|\mathscr F(\varphi^{(\ell)})(t\xi)|> c  
\end{equation}
for some positive constant $c$, where $\mathscr F(\psi)$ is the  
Fourier transform:  
$$\mathscr F(\psi)= 
\hat{\psi}(\xi)=\int_{\Bbb R^n} \psi(x)e^{-2\pi i\langle x,\xi
\rangle}\, dx, \quad \langle x,\xi\rangle=x_1\xi_1+\dots +x_n\xi_n. 
$$  
The following result in the case of the Euclidean structure is known 
(see \cite{U}).    
 
\begin{theorema}\label{theorema} 
Let $0<p\leq 1$. 
Suppose that $\varphi^{(\ell)}\in \mathscr S(\Bbb R^n)$ with  
$\int_{\Bbb R^n} \varphi^{(\ell)}\, dx=0$, $\ell=1, 2, \dots, M$, 
 and that the condition \eqref{nondegeneracy} holds.  Then 
  \begin{equation*}\label{hpequiv}
c_p\|f\|_{H^p}\leq \sum_{\ell=1}^M\|g_{\varphi^{(\ell)}}(f)\|_{p}
\leq C_p\|f\|_{H^p} 
\end{equation*}   
for $f\in H^p(\Bbb R^n)$, where $\|\cdot\|_p$ denotes the $L^p$ norm and 
 $g_{\varphi^{(\ell)}}(f)$ is defined as in \eqref{lpop} with 
 $\varphi=\varphi^{(\ell)}$, 
 $f* \varphi_t(x)=\int_{\Bbb R^n}f(x-y) \varphi_t(y)\, dy$, 
$\varphi_t(y)=t^{-n}\varphi(t^{-1}y)$. 
\end{theorema} 
See \cite{FeS2} for the Hardy space $H^p(\Bbb R^n)$. 
Analogous results for $L^p$ spaces, $1<p<\infty$, can be found in 
\cite{BCP}, \cite{H} and \cite{Sa}. 
\par 
To generalize Theorem A to the case of homogeneous groups, we note that 
 the condition \eqref{nondegeneracy} can be used to find 
 $b\in (0,1)$,  positive numbers $r_1, r_2$ with $r_1<r_2$ and 
 functions  $\eta^{(1)}, \dots, \eta^{(M)}\in \mathscr S(\Bbb R^n)$ such that 
$\supp \mathscr F(\eta^{(\ell)}) \subset \{r_1<|\xi|<r_2\}, 
1\leq \ell\leq M$,  and 
\begin{equation}\label{nondegeneracy2} 
\sum_{j=-\infty}^\infty \sum_{\ell=1}^M \mathscr F(\varphi^{(\ell)})(b^j\xi)
\mathscr F(\eta^{(\ell)})(b^j\xi) =1 
\quad \text{for $\xi\in \Bbb R^n\setminus\{0\}$.} 
\end{equation}  
See \cite[Lemma 2.1]{Sav} and also \cite[Chap. V]{ST},  \cite{CT}. 
From \eqref{nondegeneracy2} it follows that 
\begin{equation}\label{nondegeneracy3} 
 \sum_{j=-\infty}^\infty \sum_{\ell=1}^M 
\varphi^{(\ell)}_{b^j}*\eta^{(\ell)}_{b^j} = \delta \quad 
\text{in $\mathscr S'$,}  
\end{equation}  
where $\delta$ denotes the Dirac delta function. 
\par 
Also,  the condition \eqref{nondegeneracy} implies the existence of  
 functions  $\eta^{(1)}, \dots, \eta^{(M)}\in \mathscr S(\Bbb R^n)$ such that 
$\supp \mathscr F(\eta^{(\ell)}) \subset \{r_1<|\xi|<r_2\}$, 
with positive numbers $r_1, r_2$ with $r_1<r_2$, for which we have 
\begin{equation}\label{nondegeneracy4} 
 \sum_{\ell=1}^M \int_{0}^\infty 
\varphi^{(\ell)}_{t}*\eta^{(\ell)}_{t}\, \frac{dt}{t} = \delta \quad 
\text{in $\mathscr S'$.}  
\end{equation}  
\par 
Let $\Delta$ be the additive sub-semigroup of $\Bbb R$ generated by 
$0, a_1, \dots, a_n$ and let $\mathscr P_a$ be the space of polynomials 
on $\Bbb H$ of homogeneous degree less than or equal to
 $a\in \Delta$ (see Section 2.2 below for more details). 
We employ a version of  \eqref{nondegeneracy4}
 as a non-degeneracy condition    
for $\varphi^{(1)}, \dots , \varphi^{(M)}$ on $\Bbb H$ and we shall prove the 
following result analogous to Theorem A.  
\begin{theorem}\label{T} 
Let $0<p\leq 1$. We can find  $d\in \Delta$ 
with the following property. 
Suppose that  $\{\varphi^{(\ell)}\in \mathscr{S}: 
 1\leq \ell\leq M\}$ is a family of functions such that   
\begin{enumerate} 
\item[(1)] 
$$\int \varphi^{(\ell)}\, dx=0, \quad 1\leq \ell \leq M; 
$$  
\item[(2)] there exist functions $\eta^{(\ell)}\in \mathscr {S}$, 
$1\leq \ell \leq M$, satisfying that 
\begin{equation}\label{delta+} 
 \sum_{\ell=1}^M \int_{0}^\infty 
\varphi^{(\ell)}_{t}*\eta^{(\ell)}_{t}\,  \frac{dt}{t} =
 \lim_{\substack{ \epsilon\to 0, \\ B\to \infty}}  \sum_{\ell=1}^M 
\int_{\epsilon}^B 
\varphi^{(\ell)}_{t}*\eta^{(\ell)}_{t}\,  \frac{dt}{t}=\delta \quad 
\text{in $\mathscr S'$}
\end{equation} 
and that 
$$\int \eta^{(\ell)}P\, dx=0 \quad 
\text{for all $P\in \mathscr P_{d}$, \quad $1\leq \ell \leq M$.} $$  
\end{enumerate}   
Then we have 
\begin{equation}\label{e1.7+}
c_p\|f\|_{H^p}\leq  \sum_{\ell=1}^M\|g_{\varphi^{(\ell)}}(f)\|_p 
\leq C_p\|f\|_{H^p} \quad \text{for $f \in H^p$ }   
\end{equation}   
with positive constants $c_p$ and $C_p$ independent of $f$, 
where $H^p$ is the Hardy space on $\Bbb H$. 
\end{theorem}  
\par 
Let $\Bbb H$ be a stratified group with a natural dilation 
 and  let $h$ be the heat kernel on $\Bbb H$ 
(see \cite{FoS}). Define $\phi^{(j)}\in \mathscr S$, $j=1, 2, \dots$, by 
$$\phi^{(j)}(x)=\left[\partial_t^jh(x,t)\right]_{t=1}=(-L)^jh(x,1), 
  $$ 
where $\partial_t=\partial/\partial t$ and 
$L$ is the sub-Laplacian of $\Bbb H$. 
As an application of Theorem \ref{T} we have the following.  

\begin{corollary}\label{C}
    Let $f\in H^p$, $0<p\leq 1$. Then, for any $j\geq 1$, 
 we have 
$$ c_p\|f\|_{H^p}\leq \|g_{\phi^{(j)}}(f)\|_p 
\leq C_p\|f\|_{H^p}$$  
with some positive constants $c_p$, $C_p$ independent of $f$. 
\end{corollary} 
This almost retrieve Theorem 7.28 of \cite{FoS}, where the first inequality is 
shown under the condition that $f\in \mathscr S'$ vanishes weakly 
at infinity and $g_{\phi^{(j)}}(f)\in L^p$. 
\par 
As in the case of the Euclidean structure of Theorem A, 
the first inequality of \eqref{e1.7+} of the theorem 
is more difficult for us to prove than the second one;  
the second inequality can be shown by applying a theory of 
vector-valued singular integrals.  
\par 
Let 
$$S_\varphi(f)(x)=\left(\int_0^\infty \int_{\rho(x^{-1}y)<t}
|f*\varphi_t(y)|^2 t^{-\gamma-1}\, dy\, dt\right)^{1/2}  $$ 
be the Lusin area integral on the homogeneous group $\Bbb H$. 
Then in \cite{FoS}, results analogous to Theorem \ref{T} 
 were proved for $S_\varphi(f)$ 
(see \cite[Theorem 7.11 and Corollary 7.22]{FoS}), 
while the result for the Littlewood-Paley $g$ functions was 
shown only for special Littlewood-Paley functions 
$g_{\phi^{(j)}}$ associated with the heat kernel.
\par   
In \cite{Sav} an alternative proof of the first inequality of the conclusion 
of Theorem A is given by applying the Peetre maximal function $F^{**}_{N,R}$ 
 defined by 
\begin{equation*}\label{pmax}
F^{**}_{N,R}(x)=\sup_{y\in \Bbb R^n}\frac{|F(x-y)|}{(1+R|y|)^N}      
\end{equation*} 
(see \cite{P}).  The proof of \cite{Sav} is  expected 
 to extend to some other situations. Indeed, it has been applied to get the 
Littlewood-Paley function characterization of parabolic Hardy spaces of 
Calder\'{o}n-Torchinsky \cite{CT, CT2} (see \cite{Sap}); see also \cite{Sajf} 
for related results on weighted Hardy spaces. 
\par  
In this note we shall show that the methods of \cite{Sav} can be also applied 
to characterize Hardy spaces on the homogeneous groups by certain 
 Littlewood-Paley functions (Theorem \ref{T}). 
One of the ingredients of the methods is to prove a vector-valued inequality 
in Theorem \ref{T3-7} below in Section 4, 
which is stated as a weighted inequality.  
\par 
In Section 2,
we shall recall some results from \cite{FoS} needed in this note including 
the definition of Hardy spaces on $\Bbb H$,  Taylor's theorem and also 
we shall have the definition of weight classes. 
In Sections 3 and 4 
 we shall show  key estimates Lemmas \ref{L1} and \ref{L3-3}, respectively, 
  which will be used to prove Theorem \ref{T3-7} in Section 4 mentioned above. 
 The proof of Theorem \ref{T} will be completed in Section 5; 
 also the proof of Corollary \ref{C} will be given there.  
 Finally, in Section 6 we shall  employ an analogue of 
\eqref{nondegeneracy3} on $\Bbb H$ as a non-degeneracy condition  and 
we shall describe results similar to Theorems \ref{T} and
 \ref{T3-7} (Theorems \ref{TT} and \ref{TT3-7}).  
Also, we shall state discrete parameter versions of 
Theorems \ref{T} and \ref{T3-7} (Theorems \ref{Tdis} and \ref{T3-7dis}).

\section{Some preliminaries}  

In this section we have some preliminary results. See \cite{FoS} for 
results in Sections 2.1, 2.2 and 2.3. 
\subsection{Invariant derivatives} 
Let $e_j=(e^{(j)}_1,e^{(j)}_2, \dots, e^{(j)}_n)$, $1\leq j\leq n$, 
 be the element of $\Bbb H$ such that $e^{(j)}_j=1$ and $e^{(j)}_k=0$ 
if $k\neq j$.   
Define the left-invariant and right-invariant  differential operators, 
which are denoted by $X_j$ and $Y_j$, respectively,  by
\begin{gather*}  
X_jf(x)=\left[\frac{d}{dt}f(x(te_j))\right]_{t=0}, 
\\ 
Y_jf(x)=\left[\frac{d}{dt}f((te_j)x)\right]_{t=0}. 
\end{gather*}  
Then we can see that $X_j(f(A_sx))=s^{a_j}(X_jf)(A_sx)$, 
$Y_j(f(A_sx))=s^{a_j}(Y_jf)(A_sx)$.  
\par 
Let $\Bbb N_0$ denote the set of non-negative integers and 
let $I=(i_1, i_2, \dots, i_n)\in (\Bbb N_0)^n$. Define 
$$|I|=i_1+i_2+\dots +i_n, \quad a(I)= a_1i_1+a_2i_2+\dots +a_ni_n. $$  
 Higher order differential operators $X^I$ and $Y^I$ are defined as 
$$X^I=X_1^{i_1}X_2^{i_2}\dots X_n^{i_n}, \quad 
Y^I=Y_1^{i_1}Y_2^{i_2}\dots Y_n^{i_n}.  $$  
Then $|I|$ is called the order of $X^I$ and $Y^I$ and $a(I)$  
the homogeneous degree for them.  
\par 
Let $I=(i_1, i_2, \dots, i_n)$ and $I'=(i_n,\dots, i_2, i_1)$. Then 
\begin{gather*}
(X^If)*g(x)=f*(Y^{I'}g)(x), 
\\ 
\int_{\Bbb H}(X^If)(x)g(x)\, dx=(-1)^{|I|}\int_{\Bbb H}f(x)(X^{I'}g)(x)\, dx, 
\\ 
\int_{\Bbb H}(Y^If)(x)g(x)\, dx=(-1)^{|I|}\int_{\Bbb H}f(x)(Y^{I'}g)(x)\, dx, 
\\ 
X^I(f*g)(x)= (f*X^Ig)(x), \quad  Y^I(f*g)=(Y^If)*g    
\end{gather*}
for appropriate functions $f, g$. 
\subsection{Taylor polynomials} 
Let 
\begin{equation}\label{poly}
P(x)=\sum c_Ix^I, \quad x^I=x_1^{i_1}x_2^{i_2}\dots x_n^{i_n}, 
\quad I=(i_1, i_2, \dots i_n),  
\end{equation}  
be a polynomial on $\Bbb R^n$. 
We may also consider $P(x)$ as a polynomial 
on $\Bbb H$. The degree of the polynomial $P$ is 
$\max\{|I|: c_I\neq 0\}$. Also, the homogeneous degree of $P$ is 
defined to be $\max\{a(I): c_I\neq 0\}$.  
\par 
If $P(x)=x^J$, then $Y^IP$ and $X^IP$ are homogeneous of degree 
$a(J)-a(I)$ with respect to the dilation $A_t$. This implies, in particular,
 that $Y^IP=X^IP=0$ if $a(I)>a(J)$.   
\par 
Let 
 $\Delta=\{a(I): I\in (\Bbb N_0)^n\}$. 
Define  
\begin{equation}\label{abar}
\bar{a}=\min\{c\in \Delta: c>a\}.  
\end{equation}
\par 
We denote by $\mathscr P_a$ 
the space of all polynomials $P$ in \eqref{poly} with $a(I)\leq a$ 
for all $I$. 
\par 
Let $a\in \Delta$. 
Let $f$ be a function which has continuous derivatives $X^If$ 
in a neighborhood of $x\in \Bbb H$ for $a(I)\leq a$.  The left Taylor 
polynomial $P_x(y)$ of $f$ at $x$ of homogeneous degree $a$ is the unique 
polynomial $P$ such that $X^IP(0)=X^If(x)$ for all $I$ with $a(I)\leq a$.  
The right Taylor polynomial is defined similarly with $Y^I$ in place of 
$X^I$. 
\par 
We state  mean value and  Taylor inequalities.
\begin{lemma}\label{L3.1}
 Suppose that $f$ is continuously differentiable on $\Bbb H$.   
Then for $x, y\in \Bbb H$, we have 
$$|f(xy)-f(x)|\leq C\sum_{j=1}^n\rho(y)^{a_j}\sup_{\rho(z)\leq C_1\rho(y)} 
|(X_jf)(xz)|,  $$  
where the constants $C, C_1$ are independent of $x, y$ and $f$.  
\end{lemma}  
This can be shown by using Theorem 1.33 of \cite{FoS} and the relation 
$Y_j\tilde{f}=-\widetilde{X_jf}$, where $\tilde{f}(x)=f(x^{-1})$. 

\begin{lemma}\label{L2.2+} 
Let $a\in \Delta$, $a\geq 0$. Put $k=[a]$, where $[a]$ denotes the largest 
integer not exceeding $a$. 
There are constants $C_a$ and $B_a$ 
such that if $f$ is $k+1$ times continuously differentiable on $\Bbb H$, 
$x, y\in \Bbb H$ and $P_x$ is the right Taylor polynomial of $f$ at $x$ 
of homogeneous degree $a$, then 
$$|f(yx)-P_x(y)|\leq C_a\sum_{|I|\leq k+1, a(I)>a}  
 \rho(y)^{a(I)}\sup_{\rho(z)\leq B_a\rho(y)} |Y^If(zx)|.  $$
\end{lemma}  
See \cite[Theorems 1.33, 1.37]{FoS}.

\subsection{Hardy spaces} 
 We define 
$$\|\Phi\|_{(N)}=\sup_{|I|\leq N, x\in \Bbb H}(1+\rho(x))^{(N+1)(\gamma+1)}
|Y^I\Phi(x)| $$  
(see \cite[p. 35]{FoS}).  
Put 
$$B_N= \{\Phi \in\mathscr S: \|\Phi\|_{(N)}\leq 1\}.$$
Let 
$$M_{(N)}(f)(x)= \sup\{\sup_{t>0}|f*\Phi_t(x)|:\Phi\in B_N \}.  $$
The Hardy space $H^p$ on $\Bbb H$ for $p\in (0,1]$ is defined as 
$$H^p=\{f\in \mathscr S': \|f\|_{H^p}=\|M_{(N_p)}(f)\|_p <\infty\}, $$  
with sufficiently large $N_p$.  The number 
$$\min\left\{N\in \Bbb N_0: N\geq \min\{a\in \Delta: a>\gamma(p^{-1}-1)\}
 \right\} $$  
 can be taken as $N_p$, which equals $[\gamma(p^{-1}-1)]+1$ when $\Delta=
 \Bbb N_0$ (see \cite[Chap. 2]{FoS}). 
\par 
In the case of Euclidean structure, the $H^p$ spaces can be characterized by 
the radial maximal function $\sup_{t>0}|f*\varphi_t|$, where $\varphi\in 
\mathscr S$ with $\int \varphi =1$ (see \cite{FeS2}). 

\subsection{Weight functions} 

Let $B$ be a subset of $\Bbb H$. Then $B$ is a ball in $\Bbb H$ with 
center $x\in \Bbb H$ and radius $t>0$ if 
$$B=\{y\in \Bbb H: \rho(y^{-1}x)<t\}.  $$
We write $B=B(x,t)$. Let 
$\dashint_{B} f(y)\,dy=|B|^{-1}\int_{B} f(y)\,dy$, where $|B|$ denotes 
the Lebesgue measure of $B$. Let $w$ be a weight function on $\Bbb H$ and 
$1<p<\infty$.
We say that $w$ belongs to the class $A_p$ of Muckenhoupt if 
 $$  
 \sup_B \left(\dashint_B w(x)\,dx\right)\left(\dashint_B
w(x)^{-1/(p-1)}dx\right)^{p-1} < \infty, $$
where the supremum is taken over all balls $B$ in $\Bbb H$. 
\par 
The Hardy-Littlewood maximal operator is defined by 
$$M(f)(x)=\sup_{x\in B}\dashint_B|f(y)|\,dy,     $$ 
where the supremum is taken over all balls $B$ in $\Bbb H$ containing $x$. 
(See \cite{GR, GGKK}.)   
\par 
We denote by $\|f\|_{L^p_w}$ the weighted $L^p$ norm 
$$\left(\int_{\Bbb H} |f(x)|^p w(x)\, dx\right)^{1/p}. $$ 
\par
 We shall apply the following weighted vector-valued inequalities. 
\begin{lemma}\label{L2-6} 
 Let $1<\mu, \nu <\infty$. Suppose that $w\in A_\nu$. Then 
for appropriate functions $G(x,t)$ on $\Bbb H\times (0, \infty)$ we have 
$$\left(\int_{\Bbb H}\left(\int_0^\infty M(G^t)(x)^{\mu} 
\, \frac{dt}{t}\right)^{\nu/\mu} w(x)\, dx\right)^{1/\nu} 
\leq C \left(\int_{\Bbb H} \left(\int_0^\infty |G(x,t) |^{\mu}
\, \frac{dt}{t}\right)^{\nu/\mu} w(x)\, dx\right)^{1/\nu},    $$
where $G^t(x)=G(x,t)$.   
\end{lemma}  
This is a version of a result in \cite{FeS} (see  \cite[pp. 265--267]{GGKK}).

\section{Some basic estimates}   

 For $\eta, \psi \in \mathscr S$ and $t, L>0$,   let  
\begin{equation}\label{e1}
C(\eta, \psi,t,L,x)=(1+\rho(x))^{L}(\eta*\psi_{t}(x)),  \quad 
C(\eta, \psi,t,L)= \int_{\Bbb H}|C(\eta, \psi,t,L,x)|\, dx. 
\end{equation}
Define the Peetre maximal function on $\Bbb H$ by 
\begin{equation}\label{pmax}F^{**}_{N,R}(x)=\sup_{y\in \Bbb H}
\frac{|F(xy^{-1})|}{(1+R\rho(y))^N}=
\sup_{y\in \Bbb H}\frac{|F(y)|}{(1+R\rho(y^{-1}x))^N}.
\end{equation}   
Let $f\in \mathscr S'$. We say that $f$ vanishes weakly at infinity if 
$f*\phi_t\to 0$ in $\mathscr S'$ as $t\to \infty$ for all $\phi\in \mathscr S$
 (see \cite[p. 50]{FoS}).  
\begin{lemma}\label{L1}   
Suppose that $\varphi^{(\ell)}, \eta^{(\ell)} \in \mathscr S$, 
$1\leq \ell \leq M$, satisfy $\int \varphi^{(\ell)}=0$, 
$1\leq \ell \leq M$,  and \eqref{delta+}.  
Suppose that $f\in \mathscr S'$  vanishes  weakly at 
infinity and that $\psi \in \mathscr S$.  Let $b\in (0,1)$. 
Then 
\begin{equation}\label{e3-3+}
(f*\psi_t)^{**}_{L, t^{-1}}(x)\leq \sum_{\ell=1}^M
\sum_{j=-\infty}^\infty C_L b^{-Lj_+}\int_b^1 
 C(\eta^{(\ell)}, \psi, u^{-1}b^{-j},L) 
(f*\varphi^{(\ell)}_{ub^{j}t})^{**}_{L,b^{-j}t^{-1}}(x)\, \frac{du}{u},  
\end{equation} 
where $j_+=\max(0,j)$.  
\end{lemma}
\begin{proof}  
Define $\zeta\in \mathscr S$ by 
\begin{equation*}\label{zeta}
\zeta= \sum_{\ell=1}^M \int_{1}^\infty   
\varphi^{(\ell)}_{t}*\eta^{(\ell)}_{t}\,  \frac{dt}{t}.   
\end{equation*}  
The fact that $\zeta\in \mathscr S$ and $\int \zeta =1$ 
can be seen from \cite[p. 51]{FoS}. 
We have 
\begin{align*}  
f*\psi_t
&=\lim_{\substack{k\to \infty, \\ m\to \infty}}\left(f*\zeta_{tb^{m+1}}*\psi_t 
- f*\zeta_{tb^{-k}}*\psi_t\right)
\\ 
&= \lim_{\substack{k\to \infty, \\ m\to \infty}}  
\sum_{j=-k}^m \sum_{\ell=1}^M \int_b^1 
f*\varphi^{(\ell)}_{utb^{j}}*(\eta^{(\ell)}*\psi_{u^{-1}b^{-j}})_{utb^{j}}
\,  \frac{du}{u} 
\\ 
&=\sum_{j=-\infty}^\infty \sum_{\ell=1}^M \int_b^1 
f*\varphi^{(\ell)}_{utb^{j}}*(\eta^{(\ell)}*\psi_{u^{-1}b^{-j}})_{utb^{j}}
\,  \frac{du}{u}, 
\end{align*}     
if $f\in \mathscr S'$ vanishes weakly at infinity (see Proposition 1.49 
and the proof of Theorem 1.64 in \cite{FoS}).  
Noting that 
$$\eta^{(\ell)}*\psi_{u^{-1}b^{-j}}(x)
=(1+\rho(x))^{-L}C(\eta^{(\ell)}, \psi,u^{-1}b^{-j},L,x),   
$$  
we see that
\begin{multline}\label{e3} 
|f*\psi_t(z)|
\leq \sum_{\ell=1}^M 
\sum_{j=-\infty}^\infty\int_{b}^1\int |f*\varphi^{(\ell)}_{utb^{j}}(y)|
(1+t^{-1}b^{-j}\rho(y^{-1}z))^{-L} 
\\ 
\times|C(\eta^{(\ell)}, \psi,u^{-1}b^{-j}, L, A^{-1}_{utb^{j}}(y^{-1}z)|
(utb^{j})^{-\gamma}\, dy\,  \frac{du}{u},   
\end{multline} 
since $b\leq u\leq 1$ in the integral. 
We observe that
\begin{equation}\label{e4}
(1+t^{-1}b^{-j}(y^{-1}z))^{-L}(1+t^{-1}\rho(z^{-1}x))^{-L}\leq 
2^Lc_0^Lb^{-Lj_+}(1+t^{-1}b^{-j}\rho(y^{-1}x))^{-L},  
\end{equation} 
where $c_0$ is as in (7) of Section 1. 
To see this, we first note that  
\begin{multline*}
(1+t^{-1}b^{-j}\rho(y^{-1}z))(1+t^{-1}\rho(z^{-1}x))
\\ 
=b^{-j}t^{-2} 
\left(b^{j}t^2+t\rho(y^{-1}z) + b^{j}t\rho(z^{-1}x)+\rho(y^{-1}z)
\rho(z^{-1}x)\right)  
\end{multline*} 
and 
\begin{align*}
I&:=(1+t^{-1}b^{-j}\rho(y^{-1}z))(1+t^{-1}\rho(z^{-1}x))
(1+t^{-1}b^{-j}\rho(y^{-1}x))^{-1}
\\ 
&=\frac{b^{j}t^2+t\rho(y^{-1}z) + b^{j}t\rho(z^{-1}x)+
\rho(y^{-1}z)\rho(z^{-1}x)}{t(b^{j}t+\rho(y^{-1}x))} 
\\ 
&\geq \frac{b^{j}t^2+t\rho(y^{-1}z) + b^{j}t\rho(z^{-1}x)}
{t(b^{j}t+\rho(y^{-1}x))}. 
\end{align*} 
If $j\geq 0$, since $\rho(y^{-1}x)\leq c_0(\rho(y^{-1}z)+\rho(z^{-1}x))$, 
$c_0\geq 1$ and $b^j\leq 1$, 
$$I\geq b^{j}c_0^{-1}\frac{t^2+t\rho(y^{-1}x)}
{t(b^{j}t+\rho(y^{-1}x))}\geq b^{j}c_0^{-1}.   $$
Next let $j\leq 0$. If $b^jt\geq \rho(y^{-1}x)$, then 
$$I\geq \frac{b^{j}t^2+t\rho(y^{-1}z) + b^{j}t\rho(z^{-1}x)}
{2t^2b^{j}} \geq \frac{1}{2}.  $$
If $b^jt< \rho(y^{-1}x)$, since $b^j\geq 1$, 
\begin{multline*}
I\geq \frac{b^{j}t^2+t\rho(y^{-1}z) + b^{j}t\rho(z^{-1}x)}
{2t\rho(y^{-1}x)} 
\\ 
\geq \frac{b^{j}t+\rho(y^{-1}z) + \rho(z^{-1}x)}
{2\rho(y^{-1}x)}
\geq \frac{c_0^{-1}\rho(y^{-1}x)}
{2\rho(y^{-1}x)}\geq \frac{1}{2}c_0^{-1}.  
\end{multline*} 
Combining results, we can easily get \eqref{e4}. 
\par 
Multiplying both sides of \eqref{e3} by $(1+t^{-1}\rho(z^{-1}x))^{-L}$ 
and using \eqref{e4}, we have 
\begin{align*}  
&|f*\psi_t(z)|(1+t^{-1}\rho(z^{-1}x))^{-L}
\\ 
&\leq C\sum_{\ell=1}^M 
\sum_{j=-\infty}^\infty b^{-Lj_+}\int_{b}^1\int 
\frac{|f*\varphi^{(\ell)}_{utb^{j}}(y)|}{(1+t^{-1}b^{-j}\rho(y^{-1}x))^{L}} 
\\ 
& \qquad\qquad\qquad \qquad\qquad\qquad 
\times|C(\eta^{(\ell)}, \psi,u^{-1}b^{-j}, L, A^{-1}_{utb^{j}}(y^{-1}z)|
(utb^{j})^{-\gamma}\, dy  \,  \frac{du}{u}
\\ 
&\leq C\sum_{\ell=1}^M
\sum_{j=-\infty}^\infty b^{-Lj_+}\int_{b}^1 
C(\eta^{(\ell)}, \psi, u^{-1}b^{-j},L) 
(f*\varphi^{(\ell)}_{ub^{j}t})^{**}_{L,b^{-j}t^{-1}}(x)\,  \frac{du}{u}. 
\end{align*} 
The inequality \eqref{e3-3+} follows from this by taking supremum in $z$. 
\end{proof}
To estimate $C(\eta,\psi,t,L)$ in \eqref{e1} we apply the following result. 
\begin{lemma} \label{L2-4}
Let $\eta, \psi\in \mathscr S$.  
\begin{enumerate} 
\item Let $t\geq 1$. 
 Suppose that  $a\in \Delta$ and 
  $\int \eta P\ dx=0$ for all $P\in \mathscr P_a$.  
Then, for any $M\geq 0$, we have  
$$|\eta*\psi_t(x)|\leq B_1(\eta, \psi, a, M)
t^{-\bar{a}-\gamma}(1+t^{-1}\rho(x))^{-M} 
$$ 
for all $x\in \Bbb H$ with some constant $ B_1(\eta, \psi, a, M)$ 
$($see \eqref{abar} for $\bar{a})$.  
\item 
Let $0<t\leq 1$. 
If $a\in \Delta$ and   
$\int \psi P dx=0$ for all $P\in \mathscr P_a$,   
then, for any $M\geq 0$, 
$$|\eta*\psi_t(x)|\leq B_2(\eta, \psi, a, M)
t^{\bar{a}}(1+\rho(x))^{-M} $$ 
for all $x\in \Bbb H$ with some constant $ B_2(\eta, \psi, a, M)$. 
\end{enumerate}  
\end{lemma}  
\begin{proof}  
Let $t\geq 1$ to prove part (1).  
Let $P_x(y)$ be the right Taylor polynomial of $\psi$ at $x$ of homogeneous 
degree $a\in \Delta$.  
Then, if $R_x(y)=\psi(yx)-P_x(y)$, 
\begin{equation}\label{e5} 
|R_x(y)|\leq C(\psi, a, M)\rho(y)^{\bar{a}}(1+\rho(x))^{-M} 
\end{equation}  
for any $a\in \Delta$, $M>0$, provided that $\rho(x)\geq D_a\rho(y)$ with 
sufficiently large $D_a$.  This can be shown by applying Lemma \ref{L2.2+}. 
Indeed, if $D_a\geq 2c_0B_a$, $\rho(z)\leq B_a\rho(y)$ and $\rho(x)\geq 
D_a\rho(y)$, where $B_a$ is as in Lemma \ref{L2.2+}, 
 then it can be easily shown that $c_0\rho(zx)\geq \rho(x)/2$. 
 \par 
If $\int \eta P\, dx=0$ for 
$P\in \mathscr P_a$, 
$$\int \eta(y)t^{-\gamma}\psi(A_t^{-1}(y^{-1}x))\, dy 
=\int \eta(y)t^{-\gamma} R_{A_t^{-1}x}(A_t^{-1}y^{-1})\, dy=:J.   $$  
By \eqref{e5} we have 
\begin{equation}\label{e6} 
|R_{A_t^{-1}x}(A_t^{-1}y^{-1})|\leq Ct^{-\bar{a}} \rho(y)^{\bar{a}}
(1+t^{-1}\rho(x))^{-M}
\end{equation}  
if $\rho(x)\geq D_a\rho(y)$.  Let $J=J_1+J_2$, where 
$$J_1=\int\limits_{D_a\rho(y)\leq \rho(x)} 
\eta(y)t^{-\gamma} R_{A_t^{-1}x}(A_t^{-1}y^{-1})\, dy, \quad 
J_2=\int\limits_{D_a\rho(y)>\rho(x)}
 \eta(y)t^{-\gamma} R_{A_t^{-1}x}(A_t^{-1}y^{-1})\, dy. $$  
Then, \eqref{e6} implies that 
\begin{equation}\label{e7}  
|J_1|\leq Ct^{-\bar{a}-\gamma}(1+t^{-1}\rho(x))^{-M}
\int \rho(y)^{\bar{a}}|\eta(y)|\, dy 
\leq Ct^{-\bar{a}-\gamma}(1+t^{-1}\rho(x))^{-M}. 
\end{equation}  
\par 
Next we estimate $J_2$. By Lemma \ref{L2.2+}   
\begin{equation*}\label{e8+}
 |R_x(y)|\leq C(\psi,a)\sum_{|I|\leq [a]+1, a(I)>a}\rho(y)^{a(I)}, 
\end{equation*} 
 which implies that 
 $$
|R_{A_t^{-1}x}(A_t^{-1}y^{-1})|
\leq C\sum_{|I|\leq [a]+1, a(I)>a}t^{-a(I)}\rho(y)^{a(I)}
\leq Ct^{-\bar{a}}\sum_{|I|\leq [a]+1, a(I)>a}\rho(y)^{a(I)}. 
 $$  
Thus 
\begin{align}\label{e8}
|J_2|&\leq Ct^{-\bar{a}-\gamma}\int\limits_{D_a\rho(y)>\rho(x)}
|\eta(y)|\left(\sum_{|I|\leq [a]+1, a(I)>a}\rho(y)^{a(I)}\right)\, dy 
\\ 
&\leq C_{M,a} t^{-\bar{a}-\gamma}(1+\rho(x))^{-M} 
\leq C_{M,a} t^{-\bar{a}-\gamma}(1+t^{-1}\rho(x))^{-M}. \notag 
\end{align} 
 By \eqref{e7} and \eqref{e8} we have, for any $M\geq 0$, 
\begin{equation}\label{e9} 
|J|\leq Ct^{-\bar{a}-\gamma}(1+t^{-1}\rho(x))^{-M}
\end{equation}  
for $t\geq 1$.  
This completes the proof of part (1). 
\par 
To prove part (2), let $0<t\leq 1$.   We note that 
$$(\eta*\psi_t)^{\widetilde{\, }}(x)= 
s^\gamma\tilde{\psi}*\tilde{\eta}_s(A_sx), 
\quad s=t^{-1}\geq 1.  $$  
Thus by \eqref{e9}, if $M\geq 0$ and  
$\int \psi P\, dx=0$ for $P\in \mathscr P_a$,  we have, 
for $x\in \Bbb H$, 
$$ |\eta*\psi_t(x)|\leq Cs^\gamma s^{-\bar{a}-\gamma}(1+\rho(x))^{-M} 
=Ct^{\bar{a}}(1+\rho(x))^{-M}. $$
This concludes the proof. 
\end{proof} 

\begin{remark}\label{r3.3.26} 
The constants $ B_j(\eta, \psi, a, M)$, $j=1, 2$, 
 in Lemma $\ref{L2-4}$ can be taken 
independent of $\eta$ and $\psi$ if $\|\eta\|_{(L)}\leq 1$ and 
$\|\psi\|_{(L)}\leq 1$ and if $L$ is sufficiently large depending on $a, M$. 
\end{remark}

\section{Maximal function of Peetre and vector-valued inequalities}  

For the maximal function $(f*\varphi_t)_{N,t^{-1}}^{**}$ we have the estimate 
in Lemma \ref{L3-3} below. We first prove the following.  
\begin{lemma}\label{L3-2} 
 Let $F$ be continuously differentiable on $\Bbb H$.   
Let $r>0$, $N=\gamma/r$ and let $0<u\leq 1$. 
Then for $x \in \Bbb H$, we have  
$$ 
F_{N,1}^{**}(x)\leq C_r u^{-N}M(|F|^r)^{1/r}(x) +
C_r u\sum_{j=1}^n(X_jF)_{N,1}^{**}(x).   
$$   
\end{lemma}  
\begin{proof}  
For $u, r>0$ and $x, z\in \Bbb H$ we have  
\begin{align}\label{e4-1+}
|F(xz^{-1})|
&=\left(\dashint_{B(xz^{-1},u)}|F(y)+(F(xz^{-1})-F(y))|^r\, dy\right)^{1/r} 
\\  \notag 
&\leq c_r\left(\dashint_{B(xz^{-1},u)}|F(y)|^r\, dy\right)^{1/r} 
+c_r\left(\dashint_{B(xz^{-1},u)}|F(xz^{-1})-F(y)|^r\, dy\right)^{1/r},  
\end{align} 
where $c_r=1$ if $r\geq 1$ and $c_r=2^{-1+1/r}$ if $0<r< 1$. 
\par 
Let $w=xz^{-1}$, $y\in B(xz^{-1},u)$. By Lemma \ref{L3.1} 
$$|F(w)-F(y)|=|F(y(y^{-1}w))-F(y)|
\leq C\sum_{j=1}^n\rho(y^{-1}w)^{a_j}\sup_{\rho(v)\leq C_1\rho(y^{-1}w)} 
|(X_jF)(yv)|.  $$  
Since $0<u\leq 1$, 
$$|F(w)-F(y)|
\leq Cu\sum_{j=1}^n\sup_{\rho(v)\leq C_1\rho(y^{-1}w)} 
|(X_jF)(yv)|.  $$  
We note that 
$$\rho(y^{-1}xz^{-1})=\rho(y^{-1}w)<u, \quad \rho(y^{-1}x)=\rho(x^{-1}y)
\leq c_0(u +\rho(z)).  $$  
Therefore 
\begin{align*} 
\sup_{\rho(v)\leq C_1\rho(y^{-1}w)}|(X_jF)(yv)|
&\leq C\sup_{\rho(v)\leq C_1\rho(y^{-1}w)}
\frac{|(X_jF)(xx^{-1}yv)|}{\left(1+\rho(x^{-1}yv)\right)^N}(1+u+\rho(z))^N 
\\ 
&\leq C(X_jF)_{N,1}^{**}(x)(1+\rho(z))^N. 
\end{align*}  
It follows that 
\begin{equation}\label{e4-2+}
\left(\dashint_{B(xz^{-1},u)}|F(xz^{-1})-F(y)|^r\, dy\right)^{1/r} 
\leq C u\sum_{j=1}^n(X_jF)_{N,1}^{**}(x)(1+\rho(z))^N. 
\end{equation}
\par 
We observe that $B(xz^{-1},u)\subset B(x,c_0(u+\rho(z)))$, since 
we have $\rho(y^{-1}x)\leq c_0(u+\rho(z))$ if $\rho(y^{-1}(xz^{-1}))\leq u$. 
Thus 
\begin{align}\label{e4-3+}  
\left(\dashint_{B(xz^{-1},u)}|F(y)|^r\, dy\right)^{1/r} 
&\leq C\left(u^{-\gamma}(u+\rho(z))^\gamma 
\dashint_{B(x,c_0(u+\rho(z)))}|F(y)|^r\, dy\right)^{1/r} 
\\ \notag 
&\leq Cu^{-\gamma/r}(1+\rho(z))^{\gamma/r}M(|F|^r)(x)^{1/r}.  
\end{align}  
\par 
If $N=\gamma/r$, combining \eqref{e4-1+}, \eqref{e4-2+} and \eqref{e4-3+}, 
we have 
$$
\frac{|F(xz^{-1})}{(1+\rho(z))^{\gamma/r}}
\leq Cu^{-\gamma/r}M(|F|^r)(x)^{1/r}+Cu\sum_{j=1}^n(X_jF)_{N,1}^{**}(x). 
$$ 
Taking supremum in $z$, we get the conclusion.  
\end{proof} 

\begin{lemma}\label{L3-3} 
Let $N=\gamma/r$, $r>0$, $0<\delta\leq 1$.  Let $f, \varphi \in \mathscr S$. 
 Then we have 
$$ 
(f*\varphi_t)_{N,t^{-1}}^{**}(x)
\leq C_r \delta^{-N}M(|f*\varphi_t|^r)^{1/r}(x) +
C_r \delta \sum_{j=1}^n(f*(X_j\varphi)_t)_{N,t^{-1}}^{**}(x)  
$$   
for all $t>0$. 
\end{lemma} 
To prove this we apply the following. 
\begin{lemma}\label{L3-4}  
Define the operator $T_t$ by $(T_tf)(x)=f(A_tx)$. Then, for appropriate 
functions $F, f, g$ on $\Bbb H$  we have  
\begin{enumerate}  
\item  $(T_tF_{N,R}^{**})(x)=(T_tF)_{N, tR}^{**}(x)$ for all $t, N, R >0;$
\item  $T_t(f*g)(x)= t^{\gamma} (T_tf)*(T_tg)(x)$ for every $t>0;$
\item  $T_t(M(f))(x)= M(T_tf)(x)$ for every $t>0$.  
\end{enumerate}
\end{lemma} 
This can be shown by direct computation. 
\begin{proof}[Proof of Lemma $\ref{L3-3}$]  
By (1), (2) of Lemma \ref{L3-4} 
$$T_t(f*\varphi_t)_{N,t^{-1}}^{**}(x) =(T_tf*\varphi)_{N, 1}^{**}(x). $$
Using Lemmas \ref{L3-2}, we have  
$$(T_tf*\varphi)_{N, 1}^{**}(x)\leq 
C\delta^{-N} M(|T_tf*\varphi|^r)^{1/r}(x)+ 
C\delta \sum_{j=1}^n(T_tf*X_j\varphi)_{N,1}^{**}(x).  
 $$ 
 Applying $T_{t^{-1}}$ to both sides of this inequality, we can get the 
 conclusion, since  by Lemma \ref{L3-4} we have 
\begin{gather*} 
T_{t^{-1}}(T_tf*\varphi)_{N, 1}^{**}(x)=(f*\varphi_t)_{N,t^{-1}}^{**}(x), 
\\ 
T_{t^{-1}}M(|T_tf*\varphi|^r)^{1/r}(x)=M(|f*\varphi_t|^r)^{1/r}(x), 
\\ 
T_{t^{-1}}(T_tf*X_j\varphi)_{N,1}^{**}(x)
=(f*(X_j\varphi)_t)_{N,t^{-1}}^{**}(x).  
\end{gather*} 
\end{proof} 
\par 
Let $a, b,  L\geq 0$ and 
\begin{align*}
\mathscr C_{a, L}^{(1)}&=\{(\eta, \psi)\in \mathscr S\times \mathscr S: 
\sup_{t\geq 1} t^a C(\eta,\psi,t,L)<\infty\},   
\\ 
\mathscr C_{b, L}^{(2)}&=\{(\eta, \psi)\in \mathscr S\times \mathscr S: 
\sup_{0<t\leq 1} t^{-b} C(\eta,\psi,t,L)<\infty\},   
\\ 
\mathscr C_{a, b, L}&=\mathscr C_{a, L}^{(1)}\cap \mathscr C_{b, L}^{(2)}. 
\end{align*}
where $C(\eta,\psi,t,L)$ is as in \eqref{e1}. 
\par 
By Lemma $\ref{L2-4}$ we have the following results. 
\begin{remark}\label{R4-4}
Let $a, b, c, d, L, N$ be non-negative numbers and $\eta, \psi \in 
\mathscr S$. \begin{enumerate} 
\item 
If 
$\alpha\in \Delta$,   $\bar{\alpha}
\geq a+L$ and $\int \eta P\, dx =0$ for all 
$P\in \mathscr P_\alpha$, then $(\eta, \psi) \in \mathscr 
C_{a, L}^{(1)}$.  
\item  If 
$\beta\in \Delta$,  $\bar{\beta}\geq b$ and $\int \psi P\, dx =0$ 
for all $P\in \mathscr P_\beta$, then $(\eta, \psi) \in 
\mathscr C_{b, N}^{(2)}$.  
In particular, $(\eta, \psi) \in \mathscr C_{\epsilon, N}^{(2)}$ 
for some $\epsilon>0$ and for all $N$ if $\int \psi \, dx=0$. 
\item   We have $\mathscr C_{a, L}^{(j)}\subset \mathscr C_{b, L}^{(j)}$ 
if $a\geq b$ and $\mathscr C_{a, L}^{(j)}\subset \mathscr C_{a, N}^{(j)}$ 
if $L\geq N$ for $j=1, 2$.  
The set $\mathscr C_{a, b, L}$ is decreasing in each of the parameters 
$a, b, L$ when the other two are fixed.  
\end{enumerate} 
\par 
Here we give a proof of part $(1)$. Part $(2)$ can be shown similarly. 
Let $t\geq 1$. By part (1) of Lemma $\ref{L2-4}$, if $M>L+\gamma$ 
and $\bar{\alpha}\geq a+L$,  we have 
\begin{align*} 
C(\eta,\psi,t,L)&\leq Ct^{-\bar{\alpha}-\gamma +M}
\int_{\rho(x)\geq t}\rho(x)^{L-M}\, dx+ Ct^{-\bar{\alpha}-\gamma}
\int_{\rho(x)\leq t}(1+\rho(x))^{L}\, dx 
\\ 
&\leq Ct^{-\bar{\alpha}+L}\leq Ct^{-a}. 
\end{align*} 
This completes the proof. 
\end{remark}  
Using Lemmas \ref{L1} and \ref{L3-3}, we can prove the following result.  
\begin{theorem} \label{T3-5}    
Let $q\geq 1, r>0$ and $N=\gamma/r$. Let $\varphi^{(\ell)}\in \mathscr S$, 
$\int \varphi^{(\ell)}=0$, 
$1\leq \ell\leq M$. 
Suppose that there exist $\eta^{(\ell)} \in \mathscr S$, $1\leq \ell\leq M$, 
for which we have \eqref{delta+}. 
Let $f\in \mathscr S$. 
If $(\eta^{(m)}, X_k\varphi^{(\ell)})
\in \mathscr C_{N+\epsilon, N}^{(1)}$ with some $\epsilon>0$ for all 
$k=1, \dots, n$ and $\ell, m=1, \dots, M$,
then 
\begin{equation}\label{e4-4+} 
\sum_{\ell=1}^M\int_0^\infty (f*\varphi^{(\ell)}_t)_{N, t^{-1}}^{**}(x)^q
\, \frac{dt}{t} 
\leq C\sum_{\ell=1}^M\int_0^\infty M(|f*\varphi^{(\ell)}_t|^r)(x)^{q/r}
\, \frac{dt}{t}. 
\end{equation} 
\end{theorem} 
\begin{proof} 
By the assumption of the theorem and (2), (3) of Remark \ref{R4-4} we have 
$(\eta^{(m)}, X_k\varphi^{(\ell)})
\in \mathscr C_{N+\epsilon, \epsilon, N}$ for some $\epsilon>0$.  
Thus by \eqref{e3-3+} of Lemma \ref{L1}  we have 
\begin{align*}
 &(f*(X_k\varphi^{(\ell)})_t)^{**}_{N, t^{-1}}(x) 
\\ 
&\leq 
 \sum_{m=1}^M\sum_{j=-\infty}^\infty C_N b^{-Nj_+} \int_b^1  
C(\eta^{(m)}, X_k\varphi^{(\ell)}, u^{-1}b^{-j},N) 
(f*\varphi^{(m)}_{ub^{j}t})^{**}_{N,b^{-j}t^{-1}}(x)\, \frac{du}{u} 
\\ 
&\leq 
 \sum_{m=1}^M\sum_{j=-\infty}^\infty C_{N,b} b^{\epsilon|j|} \int_b^1  
(f*\varphi^{(m)}_{ub^{j}t})^{**}_{N,b^{-j}t^{-1}}(x)\, \frac{du}{u}.   
\end{align*}  
Using this and Lemma \ref{L3-3}, we see that  
\begin{multline*}
(f*\varphi^{(\ell)}_t)_{N,t^{-1}}^{**}(x)
\leq C \delta^{-N}M(|f*\varphi^{(\ell)}_t|^r)^{1/r}(x) 
\\ 
+C\delta \sum_{m=1}^M\sum_{j=-\infty}^\infty C_{N,b} b^{\epsilon|j|} \int_b^1  
(f*\varphi^{(m)}_{ub^{j}t})^{**}_{N,b^{-j}t^{-1}}(x)\, \frac{du}{u}. 
\end{multline*} 
Thus, applying H\"{o}lder's inequality when $q>1$,  we have 
\begin{multline*}
(f*\varphi^{(\ell)}_t)^{**}_{N, t^{-1}}(x)^q\leq C\delta^{-Nq} 
M(|f*\varphi^{(\ell)}_t|^r)(x)^{q/r} 
\\ 
+  C_{N,b,q,M}\delta^q\sum_{m=1}^M\sum_{j=-\infty}^\infty 
b^{q\epsilon|j|/2}  
\left(\int_b^1  
(f*\varphi^{(m)}_{ub^{j}t})^{**}_{N,b^{-j}t^{-1}}(x)\, \frac{du}{u}\right)^q.  
\end{multline*} 
Since $q\geq 1$, H\"{o}lder's inequality implies that 
\begin{equation}\label{e4-5+}
\left(\int_b^1  
(f*\varphi^{(m)}_{ub^{j}t})^{**}_{N,b^{-j}t^{-1}}(x)\, \frac{du}{u}\right)^q  
\leq (\log(1/b))^{q/q'}\int_b^1  
(f*\varphi^{(m)}_{ub^{j}t})^{**}_{N,b^{-j}t^{-1}}(x)^q\, \frac{du}{u}. 
\end{equation}  
So we see that  
\begin{multline}\label{e3-1}
(f*\varphi^{(\ell)}_t)^{**}_{N, t^{-1}}(x)^q\leq C\delta^{-Nq} 
M(|f*\varphi^{(\ell)}_t|^r)(x)^{q/r} 
\\ 
+  C_{N,b,q,M}\delta^q\sum_{m=1}^M\sum_{j=-\infty}^\infty 
b^{q\epsilon|j|/2}  
\int_b^1  
(f*\varphi^{(m)}_{ub^{j}t})^{**}_{N,b^{-j}t^{-1}}(x)^q\, \frac{du}{u}.  
\end{multline} 
By integration of both sides of the inequality \eqref{e3-1} over $(0,\infty)$ 
with respect to the measure $dt/t$,  it follows that 
\begin{align*}\label{} 
&\sum_{\ell=1}^M
\int_0^\infty (f*\varphi^{(\ell)}_t)^{**}_{N, t^{-1}}(x)^q\, \frac{dt}{t} 
\leq 
C\delta^{-Nq} \sum_{\ell=1}^M\int_0^\infty M(|f*\varphi^{(\ell)}_t|^r)(x)^{q/r}
\, \frac{dt}{t} 
\\ 
&+ C\delta^q \left[\sum_{j=-\infty}^\infty 
b^{q\epsilon|j|/2}\right] \sum_{\ell=1}^M  
\int_b^1\int_0^\infty (f*\varphi^{(\ell)}_{t})^{**}_{N,ut^{-1}}(x)^q  
\, \frac{dt}{t} \, \frac{du}{u} 
\\ 
&\leq 
C\delta^{-Nq} \sum_{\ell=1}^M\int_0^\infty M(|f*\varphi^{(\ell)}_t|^r)(x)^{q/r}
\, \frac{dt}{t} 
\\ 
&+ C\delta^q \left(\int_b^1 u^{-Nq} \, \frac{du}{u}\right)
\left[\sum_{j=-\infty}^\infty 
b^{q\epsilon|j|/2}\right] \sum_{\ell=1}^M  
\int_0^\infty (f*\varphi^{(\ell)}_{t})^{**}_{N,t^{-1}}(x)^q  \, \frac{dt}{t},  
\end{align*}  
where  the inequality 
$$(f*\varphi^{(\ell)}_{t})^{**}_{N,ut^{-1}}(x)\leq u^{-N}
(f*\varphi^{(\ell)}_{t})^{**}_{N,t^{-1}}(x)   $$ 
has been used. 
The inequality \eqref{e4-4+} follows from this by taking $\delta$ 
sufficiently small, 
since the last sum of integrals is finite, which can be easily seen under 
the conditions that $f, \varphi^{(\ell)}\in \mathscr S$ and 
$\int \varphi^{(\ell)}\, dx =0$. 

\end{proof} 
 We have some vector-valued inequalities, which are stated in  more 
general forms as weighted inequalities than needed in proving Theorem \ref{T}. 

\begin{theorem}\label{T3-7}      
 Let $N>0$,   $\gamma/N<p, q<\infty, q\geq 1$ and $w \in A_{pN/\gamma}$. 
Let  $\varphi^{(\ell)} \in \mathscr S$, $\int \varphi^{(\ell)}\, dx=0$, 
$1\leq \ell\leq M$.  
Suppose that there exist $\eta^{(\ell)} \in \mathscr S$, 
$1\leq \ell\leq M$, for which we have \eqref{delta+}. 
 Also, suppose that $(\eta^{(m)}, X_k\varphi^{(\ell)})
\in \mathscr C_{N+\epsilon, N}^{(1)}$ with some $\epsilon>0$ for  
$k=1, \dots, n$ and $\ell, m=1, \dots, M$.   
Let $\psi\in \mathscr S$ and $\int \psi \, dx=0$. 
Suppose that $(\eta^{(\ell)}, \psi) \in 
\mathscr C_{N+\epsilon, N}^{(1)}$ for some $\epsilon>0$ for 
$1\leq \ell\leq M$.  Let $f\in \mathscr S$. 
 Then we have 
\begin{equation*}\label{e4-6+}  
 \left\|\left(\int_0^\infty\left((f*\psi_t)_{N,t^{-1}}^{**}\right)^q
\, \frac{dt}{t}\right)^{1/q}\right\|_{L^p_w} 
\leq C\sum_{\ell=1}^M \left\|\left(\int_0^\infty|f*\varphi^{(\ell)}_t|^q\, 
\frac{dt}{t}\right)^{1/q}\right\|_{L^p_w}    
\end{equation*}   
 with a positive constant $C$ independent of $f$.    
\end{theorem}  

\begin{proof}   
Since, as in the proof of Theorem \ref{T3-5}, 
$(\eta^{(\ell)}, \psi) \in \mathscr C_{N+\epsilon, \epsilon, N}$, 
$1\leq \ell\leq M$, for some $\epsilon>0$,   
by \eqref{e3-3+} of Lemma \ref{L1} we have 
\begin{align*}
 (f*\psi_t)^{**}_{N, t^{-1}}(x)&\leq 
 \sum_{\ell=1}^M\sum_{j=-\infty}^\infty C_N b^{-Nj_+}  
\int_b^1 C(\eta^{(\ell)}, \psi, u^{-1}b^{-j},N) 
(f*\varphi^{(\ell)}_{ub^{j}t})^{**}_{N,b^{-j}t^{-1}}(x) \, \frac{du}{u}
\\ 
&\leq 
\sum_{\ell=1}^M \sum_{j=-\infty}^\infty C_{N,b} b^{\epsilon|j|}  
\int_b^1(f*\varphi^{(\ell)}_{ub^{j}t})^{**}_{N,b^{-j}t^{-1}}(x)
\, \frac{du}{u}.   
\end{align*}  
Thus, as in the proof of Theorem \ref{T3-5}, we can get 
\begin{equation}\label{e3-2} 
\int_0^\infty (f*\psi_t)^{**}_{N, t^{-1}}(x)^q\, \frac{dt}{t} 
\leq  C \sum_{\ell=1}^M  
\int_0^\infty (f*\varphi^{(\ell)}_{t})^{**}_{N,t^{-1}}(x)^q  \, \frac{dt}{t}. 
\end{equation}  
\par  
Let $r=\gamma/N<q, p<\infty$ and $w\in A_{pN/\gamma}$.  
Then by \eqref{e3-2}, Theorem \ref{T3-5} and Lemma \ref{L2-6},  
it follows that 
\begin{align*}\label{ineq11} 
\left(\int_{\Bbb H} \left(\int_0^\infty \right.\right. & \left.\left.
\left((f*\psi_t)_{N,t^{-1}}^{**}(x)\right)^q\,  \frac{dt}{t}
\right)^{p/q} w(x)\, dx\right)^{1/p} 
\\
&\leq C\left(\int_{\Bbb H}\left(\sum_{\ell=1}^M \int_0^\infty 
 (f*\varphi^{(\ell)}_{t})^{**}_{N,t^{-1}}(x)^q \, \frac{dt}{t}
\right)^{p/q} w(x)\, dx\right)^{1/p} 
\\
&\leq C\sum_{\ell=1}^M \left(\int_{\Bbb H}\left(\int_0^\infty 
M(|f*\varphi^{(\ell)}_t|^r)(x)^{q/r} 
\, \frac{dt}{t}\right)^{p/q} w(x)\ dx\right)^{1/p}    \notag 
\\
&\leq C\sum_{\ell=1}^M \left(\int_{\Bbb H} 
\left(\int_0^\infty |(f*\varphi^{(\ell)}_t)(x)|^{q}
\, \frac{dt}{t}\right)^{p/q} w(x)\, dx\right)^{1/p}.    \notag 
\end{align*}  
This completes the proof of Theorem \ref{T3-7}.  
\end{proof}

\section{Proofs of Theorem $\ref{T}$ and Corollary $\ref{C}$}  

 Suppose that $\Psi\in \mathscr S$ and $\int \Psi \, dx=0$.  
Let $\epsilon\in (0,1)$ and 
$$S_{\Psi,\epsilon}(h)(x)=\int_\epsilon^{\epsilon^{-1}}h(\cdot,t)*
\Psi_t(x)\, \frac{dt}{t},  $$  
for appropriate functions $h$ on $\Bbb H\times (0,\infty)$. 
Let $\mathscr H$ be the Hilbert space of 
functions $\ell(t)$ on $(0,\infty)$ such that $\|\ell\|_{\mathscr H}=
\left(\int_0^\infty|\ell(t)|^2\, dt/t\right)^{1/2}<\infty$.  
Let  $H^p_{\mathscr H}$ be the Hardy spaces of distributions on $\Bbb H$  
with values in $\mathscr H$ and 
let  $L^2_{\mathscr H}$ be the $L^2$ space of functions on $\Bbb H$  
with values in $\mathscr H$.  
\par 
We state some lemmas for the proof of Theorem \ref{T}. 
\begin{lemma}\label{L4-1} Let $0 <p\leq 1$. 
If $h\in H^p_{\mathscr H}\cap L^2_{\mathscr H}$, then 
$$\sup_{\epsilon\in (0,1)}
\|S_{\Psi,\epsilon}(h)\|_{H^p}\leq C\|h\|_{H^p_{\mathscr H}},   $$  
where $C$ is a constant independent of $h$.  
\end{lemma} 
\begin{proof} 
First we show that 
$$\|S_{\Psi,\epsilon}(h)\|_{2}\leq C\|h\|_{L^2_{\mathscr H}}.   $$  
To see this, we note that 
$$\int S_{\Psi,\epsilon}(h)(x)g(x)\, dx=\int_\epsilon^{\epsilon^{-1}}\int 
h(y,t) g*\tilde{\Psi}_t(y) \, dy\, \frac{dt}{t}.    $$ 
So, Schwarz's inequality implies that 
$$\left|\int S_{\Psi,\epsilon}(h)(x)g(x)\, dx\right|\leq 
\|h\|_{L^2_{\mathscr H}}\left(\int\int_0^\infty |g*\tilde{\Psi}_t(y)|^2
\, dy\, \frac{dt}{t}\right)^{1/2}.   $$   
It is known that 
$$\left(\int\int_0^\infty |g*\tilde{\Psi}_t(y)|^2
\, dy\, \frac{dt}{t}\right)^{1/2}\leq C\|g\|_2. $$ 
(See \cite[pp. 223--224]{FoS}.) 
Thus the result follows from the converse of H\"{o}lder's inequality. 
  \par 
  Since $Y^{I}\Psi_t=t^{-a(I)}(Y^I\Psi)_t$, we easily see that 
 $$\int_0^\infty \left|Y^I\Psi_t(x)\ell(t)\right|\, \frac{dt}{t}\leq  
 \left\|Y^I\Psi_t(x)\right\|_{\mathscr H} \|\ell\|_{\mathscr H} 
\leq  C\|\ell\|_{\mathscr H}\rho(x)^{-\gamma-a(I)}.  $$  
Thus if we define $K:\Bbb H\to \mathscr B=\mathscr B(\mathscr H, \Bbb C)$ 
(the space of bounded linear operators from $\mathscr H$ to $\Bbb C$)  
by  
$$K(x)\ell =\int_\epsilon^{\epsilon^{-1}} \Psi_t(x)\ell(t)\, \frac{dt}{t}, $$ 
then 
$$\|K(x)\|_{\mathscr B}\leq C\rho(x)^{-\gamma}, \quad  
\|Y^IK(x)\|_{\mathscr B}\leq C\rho(x)^{-\gamma-a(I)}. $$
 Therefore the conclusion of the lemma follows from a vector-valued version of 
 \cite[Theorem 6.10]{FoS} (see \cite[Theorem 6.20]{FoS}). 
 \end{proof}

 Also, we need the following result in proving the theorem.
\begin{lemma}\label{L5-2}  
Let $A$ be a non-negative integer.  We can find functions 
$U^{(\ell)}, V^{(\ell)}\in \mathscr{S}$, $1\leq \ell\leq M$,  
such that 
\begin{enumerate} 
\item[(1)] $$\int U^{(\ell)}P\, dx= \int V^{(\ell)}P\, dx=0 $$  
for all $P\in \mathscr P_A;$   
\item[(2)] $U^{(\ell)}=u^{(\ell)}*v^{(\ell)}$, with 
$u^{(\ell)}, v^{(\ell)}\in \mathscr{S}(\Bbb R^n)$ satisfying 
$$\int u^{(\ell)}P\, dx= \int v^{(\ell)} P\, dx=0 $$  
for all $P\in \mathscr P_A;$  
\item[(3)] $\supp(V^{(\ell)}) \subset B(0,1);$ 
\item[(4)] 
$$\sum_{\ell=1}^M \int_0^\infty
\left(U^{(\ell)}_{t}*V^{(\ell)}_{t}\right)\, \frac{dt}{t}=
\lim_{\substack{ \epsilon\to 0, \\ B\to \infty}} 
\sum_{\ell=1}^M \int_\epsilon^B  
\left(U^{(\ell)}_{t}*V^{(\ell)}_{t}\right)\, \frac{dt}{t}=\delta 
\quad \text{in $\mathscr S'$.}$$
\end{enumerate}   

\end{lemma}  

\begin{proof}
This follows from \cite[Theorem 1.62]{FoS} except for the vanishing moment  
property of $v^{(\ell)}$ in (2), 
which can be easily shown as follows by using Lemma 1.60 of \cite{FoS}. 
Let $L$ be a sufficiently large number with $L\geq A$. 
By the remark above, for this $L$ 
we have functions  $\Phi^{(\ell)}, \Psi^{(\ell)}\in 
\mathscr{S}$, $1\leq \ell\leq H$,  such that 
\begin{enumerate} 
\item[$\bullet$] $$\int \Phi^{(\ell)}P\, dx= \int \Psi^{(\ell)}P\, dx=0 
\quad \text{for all $P\in \mathscr P_L;$}   
$$  
\item[$\bullet$] $\Phi^{(\ell)}=\varphi^{(\ell)}*\alpha^{(\ell)}$, with 
$ \alpha^{(\ell)}\in \mathscr{S}$ and $\varphi^{(\ell)}\in \mathscr{S}$ 
satisfying 
$$\int \varphi^{(\ell)}P\, dx=0 \quad 
\text{for all $P\in \mathscr P_L;$  
} $$  
\item[$\bullet$] $\Psi^{(\ell)}$ is supported on $B(0,1);$ 
\item[$\bullet$] 
$$ \sum_{\ell=1}^H \int_0^\infty 
\Phi^{(\ell)}_{t}*\Psi^{(\ell)}_{t}=\delta 
\quad \text{in $\mathscr S'$.}$$
\end{enumerate}  
If $L$ is sufficiently large, then by Lemma 1.60 of \cite{FoS}  
$\varphi^{(\ell)}$ can be written as 
$$\varphi^{(\ell)}=\sum_{A+1\leq a(J)\leq a_n(A+1)} X^J\phi_{J,\ell}, $$
with $\phi_{J,\ell}\in \mathscr S$ satisfying 
$$\int \phi_{J,\ell}P\, dx =0 \quad \text{for all $P\in \mathscr P_A.$}   
$$ 
Thus, using a result of Section 2.1, we see that 
\begin{equation}\label{e3-1+}
\Phi^{(\ell)}=\sum_{A+1\leq a(J)\leq a_n(A+1)} 
X^J\phi_{J,\ell}*\alpha^{(\ell)} =\sum_{A+1\leq a(J)\leq a_n(A+1)} 
\phi_{J,\ell}*Y^{J'}\alpha^{(\ell)}. 
\end{equation}  
Since $a(J)\geq A+1$, we have 
$$\int (Y^{J'}\alpha^{(\ell)}) P\, dx 
=(-1)^{|J|}\int \alpha^{(\ell)}Y^{J} P\, dx=0 
\quad \text{for all $P\in \mathscr P_A.$}   
$$ 
We rewrite the expression of $\Phi^{(\ell)}$ in \eqref{e3-1+} as 
\begin{equation*}\label{e3-2+}
\Phi^{(\ell)}=\sum_{k=1}^{K} 
\phi_{k,\ell}*\alpha_{k,\ell}  
\end{equation*}  
with 
$$ \int \phi_{k,\ell} P\, dx =\int \alpha_{k,\ell} P\, dx =0 \quad 
\text{for all $P\in \mathscr P_A.$}   $$
 Then  
$$\Phi^{(\ell)}*\Psi^{(\ell)}=\sum_{k=1}^{K} 
(\phi_{k,\ell}*\alpha_{k,\ell})*\psi_{k,\ell},   $$  
where $\psi_{k,\ell}=\Psi^{(\ell)}$ for $1\leq k\leq K$.  
By this decomposition, obviously we obtain the desired result. 
\end{proof}

By Lemma \ref{L4-1} we have the following. 
\begin{lemma}\label{L4-2} 
Let $U^{(\ell)}$, $1\leq \ell\leq M$, be functions in $\mathscr{S}$ 
with $\int U^{(\ell)}\, dx=0$ for 
which there exist $V^{(\ell)}\in \mathscr{S}$, $1\leq \ell\leq M$, such that 
 $\int V^{(\ell)}\, dx=0$ and 
$$\sum_{\ell=1}^M \int_0^\infty
\left(U^{(\ell)}_{t}*V^{(\ell)}_{t}\right)\, \frac{dt}{t}=
\lim_{\substack{ \epsilon\to 0, \\ B\to \infty}} 
\sum_{\ell=1}^M \int_\epsilon^B  
\left(U^{(\ell)}_{t}*V^{(\ell)}_{t}\right)\, \frac{dt}{t}=\delta 
\quad \text{in $\mathscr S'$.}$$
Suppose that $f\in \mathscr S\cap H^p$, $0<p\leq 1$. 
Put $h^{(\ell)}(y,t)=f*U^{(\ell)}_t(y)$.  Then,  
$h^{(\ell)}\in H^p_{\mathscr H}$ and 
$$\|f\|_{H^p}\leq C\sum_{\ell=1}^M\|h^{(\ell)}\|_{H^p_{\mathscr H}}. $$
\end{lemma} 
\begin{proof} 
The fact that $h^{(\ell)}\in H^p_{\mathscr H}$ can be shown as in the proof of 
Lemma \ref{L4-1} by a theory of vector-valued singular integrals 
(see \cite[Chap. 7]{FoS}).   
Let $\psi \in \mathscr S$. 
If $f\in \mathscr S\cap H^p$,  by Theorem 1.64 of \cite{FoS}  
$$\sum_{\ell=1}^M\int_\epsilon^{\epsilon^{-1}}f*(U^{(\ell)}*V^{(\ell)})_t\, 
\frac{dt}{t}*\psi_s 
\to f*\psi_s \quad \text{as $\epsilon\to 0$.} $$   
It follows that 
$$|f*\psi_s| \leq \liminf_{\epsilon \to 0} \sum_{\ell=1}^M\sup_{s>0} 
\left|\int_\epsilon^{\epsilon^{-1}}f*U^{(\ell)}_t*V^{(\ell)}_t \, 
\frac{dt}{t} *\psi_s\right|. 
$$
Taking $h(y,t)=f*U^{(\ell)}_t(y)$ and $\Psi=V^{(\ell)}$ in Lemma \ref{L4-1},  
  we see that 
\begin{multline*} 
\int \sup_{s>0, \psi\in B_{N_p}} 
\left|\int_\epsilon^{\epsilon^{-1}}f*U^{(\ell)}_t*V^{(\ell)}_t
\, \frac{dt}{t} *\psi_s \right|^p \, dx 
\\ 
\leq C 
\int \sup_{s>0, \phi\in B_{N_p}} \left(\int_0^{\infty}
|f*U^{(\ell)}_t*\phi_s|^2 \, \frac{dt}{t} \right)^{p/2} \, dx      
\end{multline*}  
for sufficiently large $N_p$.   Therefore by Fatou's lemma we have 
\begin{align*} 
\int \sup_{s>0, \psi\in B_{N_p}} |f*\psi_s|^p\, dx 
& \leq C \liminf_{\epsilon \to 0} 
 \sum_{\ell=1}^M\int \sup_{s>0, \psi\in B_{N_p}} 
\left|\int_\epsilon^{\epsilon^{-1}}f**U^{(\ell)}_t*V^{(\ell)}_t
\, \frac{dt}{t} *\psi_s \right|^p \, dx 
\\ 
&\leq C\sum_{\ell=1}^M\int \sup_{s>0, \phi\in B_{N_p}} \left(\int_0^{\infty}
|f*U^{(\ell)}_t*\phi_s|^2 \, \frac{dt}{t} \right)^{p/2} \, dx,  
\end{align*} 
which implies the conclusion, if $N_p$ is sufficiently large. 

\end{proof}   

In proving the theorem we combine Lemmas \ref{L5-2} and \ref{L4-2} 
with the following result. 
\begin{lemma}\label{L4-3}  
Let $f\in \mathscr S'$ and $N>0$. Then there exist $L>0$ and 
$a \in \Delta$ such that if 
 $\Phi=\psi*\alpha$, $\psi, \alpha\in \mathscr S$ with 
$\int \alpha P \, dx=0$ for all $P\in \mathscr P_a$,  
then 
$$ \sup_{s>0, \phi\in B_L} |f*\Phi_t*\phi_s| \leq 
C(f*\psi_t)_{N, t^{-1}}^{**} $$  
with some constant $C$ depending only on $\|\alpha\|_{(L)}, 
\|\phi\|_{(L)},N$.  
\end{lemma} 
\begin{proof} 
To prove this we first note that 
\begin{equation}\label{e4-1} 
C(\alpha,\phi,u,N)=
\int (1+\rho(y))^N|\alpha*\phi_u(y)|\, dy \leq C \quad \text{for $u>0$,} 
\end{equation} 
if $\phi, \alpha\in B_{L}$ and $\int \alpha P \, dx=0$ when 
$P\in \mathscr P_a$ for some $L, a$.  
This can be seen as follows. 
 If $u\in (0,1]$, 
\begin{align*}
(1+\rho(y))^{M}|\alpha*\phi_u(y)| &\leq C\int (1+\rho(yz^{-1}))^M
|\alpha(yz^{-1})|(1+\rho(z))^M|\phi_u(z)|\, dz 
\\ 
&\leq C\|\alpha\|_{(M/\gamma)}\int (1+u\rho(z))^M|\phi(z)|\, dz
\\ 
&\leq C\|\alpha\|_{(M/\gamma)}\int (1+\rho(z))^M|\phi(z)|\, dz 
\\ 
&\leq C\|\alpha\|_{(M/\gamma)}\int (1+\rho(z))^{-\gamma-1}
(1+\rho(z))^{(\gamma^{-1}M+1)(\gamma+1)}|\phi(z)|\, dz 
\\ 
&\leq C \|\alpha\|_{(M/\gamma)}\|\phi\|_{(M/\gamma)}. 
\end{align*}
This implies \eqref{e4-1} for $u\in (0,1]$ if $M\geq N+\gamma+1$.  
\par 
 Next, let $u>1$.   Then, \eqref{e4-1} follows from 
  (1) of Remark \ref{R4-4} and its  
proof along  with  Remark \ref{r3.3.26}. 
\par 
Using \eqref{e4-1}, we see that 
\begin{align*}  
\left|f*\Phi_t*\phi_s(x)\right|&=
\left|f*\psi_t*\alpha_t*\phi_s(x)\right|=\left|\int (f*\psi_t)(xy^{-1})
(\alpha_t*\phi_s)(y)\, dy\right|
\\ 
&\leq (f*\psi_t)_{N, t^{-1}}^{**}(x)\int (1+t^{-1}\rho(y))^N
|\alpha_t*\phi_s(y)|\, dy 
\\ 
&= (f*\psi_t)_{N, t^{-1}}^{**}(x)\int (1+\rho(y))^N
|\alpha*\phi_{s/t}(y)|\, dy  
\\ 
&\leq C(f*\psi_t)_{N, t^{-1}}^{**}(x). 
\end{align*} 
This implies the conclusion. 
\end{proof} 

\begin{proof}[Proof of Theorem $\ref{T}$]  
Since the second inequality of the conclusion  \eqref{e1.7+} is shown in 
\cite{FoS}, it remains only to prove the first inequality for some 
$d \in \Delta$.  
Let $f\in \mathscr S \cap H^p$, $0<p\leq 1$. Let $N>\gamma/p$. 
Let $U^{(\ell)}$, $1\leq \ell\leq M'$, 
 be as in Lemma \ref{L4-2} with $M'$ in place of $M$.  
 Then, by Lemma \ref{L4-2} we
have, for sufficiently large $N_p$,   
$$ 
\|f\|_{H^p}^p \leq C\sum_{\ell=1}^{M'}\int \sup_{s>0, \phi\in B_{N_p}} 
\left(\int_0^{\infty}
|f*U^{(\ell)}_t*\phi_s|^2 \, \frac{dt}{t} \right)^{p/2} \, dx. 
$$ 
By Lemma \ref{L5-2} we can find such $U^{(\ell)}$ and 
we may assume that $U^{(\ell)}=u^{(\ell)}*v^{(\ell)}$ as 
in Lemma \ref{L5-2} (2).  
For $v^{(\ell)}$, we use Lemma \ref{L5-2} (2) with a number $A$ large enough 
and we apply the property $\int u^{(\ell)}\, dx=0$ for $u^{(\ell)}$. 
If $A$ of Lemma \ref{L5-2} (2) and $N_p$ are sufficiently large, 
by Lemma \ref{L4-3} we have
\begin{multline*} 
\sum_{\ell=1}^{M'}\int \sup_{s>0, \phi\in B_{N_p}} \left(\int_0^{\infty}
|f*U^{(\ell)}_t*\phi_s|^2 \, \frac{dt}{t} \right)^{p/2} \, dx 
\\ 
\leq C\sum_{\ell=1}^{M'}  
\left\|\left(\int_0^\infty\left((f*u^{(\ell)}_t)_{N,t^{-1}}^{**}\right)^2
\, \frac{dt}{t}\right)^{1/2}\right\|_{p}^p,  
\end{multline*}
 which implies 
$$\|f\|_{H^p}^p \leq C\sum_{\ell=1}^{M'}  
\left\|\left(\int_0^\infty\left((f*u^{(\ell)}_t)_{N,t^{-1}}^{**}\right)^2
\, \frac{dt}{t}\right)^{1/2}\right\|_{p}^p.    
$$  
Combining this with Theorem \ref{T3-7} with $\psi=u^{(\ell)}$, 
 $q=2$ and $w=1$ and recalling Remark \ref{R4-4}, 
if $d$ of Theorem \ref{T} is sufficiently large, we get the first inequality 
of \eqref{e1.7+} for 
$f\in \mathscr S \cap H^p$. So we have  \eqref{e1.7+} for 
$f\in \mathscr S \cap H^p$, from which we can deduce \eqref{e1.7+} for 
general $f\in H^p$, since $\mathscr S \cap H^p$ is dense in 
$H^p$ (see \cite{FoS}). 
This completes the proof of Theorem \ref{T}. 
\end{proof} 

\begin{proof}[Proof of Corollary $\ref{C}$] 
Let $\phi^{(j)}$, $j=1, 2, \dots$,  be as in the corollary. Then it is known 
that 
\begin{equation*}
\int_{\Bbb H}\phi^{(j)}(x)P(x)\, dx=0 \quad 
\text{for all $P\in \mathscr P_{2j-1}$}
\end{equation*} 
and 
\begin{equation*}
c_{jk}\int_{0}^\infty \phi^{(j)}_t*\phi_t^{(k)}\, \frac{dt}{t}=\delta 
\end{equation*}  
for all positive integers $j, k$ with some non-zero constant $c_{jk}$ 
(see \cite[Chap. 7]{FoS}).  
Thus we can apply Theorem \ref{T} with $M=1$, $\varphi^{(1)}= \phi^{(j)}$ 
and $\eta^{(1)}=\phi^{(k)}$, taking a sufficiently large number $k$, 
to get the desired result.

\end{proof}

\section{Another formulation for non-degeneracy} 

In this section we employ a version of 
\eqref{nondegeneracy3} as a non-degeneracy condition. We first  
 state results analogous to Theorems \ref{T} and \ref{T3-7}. 
\begin{theorem}\label{TT} 
Let $0<p\leq 1$. There exists  $d \in \Delta$  
with the following property.  If 
  $\{\varphi^{(\ell)}\in \mathscr{S}:1\leq \ell\leq M\}$ 
is a family of functions such that   
\begin{enumerate} 
\item[(1)] 
$$\int \varphi^{(\ell)}\, dx=0,  \quad 1\leq \ell\leq M; $$  
\item[(2)] 
\begin{equation}\label{deltad} 
\sum_{j=-\infty}^\infty \sum_{\ell=1}^M
\varphi^{(\ell)}_{b^{j}}*\eta^{(\ell)}_{b^{j}}=
\lim_{\substack{ k\to \infty, \\ m\to \infty}} \sum_{j=-k}^m  \sum_{\ell=1}^M 
\varphi^{(\ell)}_{b^{j}}*\eta^{(\ell)}_{b^{j}}=\delta \quad 
\text{in $\mathscr S'$}
\end{equation} 
for some $b\in (0,1)$ with some $\eta^{(\ell)}\in \mathscr {S}$, 
$1\leq \ell \leq M$, satisfying that 
$$\int \eta^{(\ell)}P\, dx=0 \quad 
\text{for all $P\in \mathscr P_{d}$, \quad $1\leq \ell \leq M$}.  $$  
\end{enumerate}   
Then we have 
$$c_p\|f\|_{H^p}\leq  \sum_{\ell=1}^M\|g_{\varphi^{(\ell)}}(f)\|_p 
\leq C_p\|f\|_{H^p} \quad \text{for $f \in H^p$, }   $$  
where $c_p$ and $C_p$ are positive constants independent of $f$. 
\end{theorem}  

\begin{theorem}\label{TT3-7}    
Suppose that $\varphi^{(\ell)} \in \mathscr S$, 
 $\int \varphi^{(\ell)}\, dx=0$, $1\leq \ell\leq M$,  and that 
we can find $\eta^{(\ell)} \in \mathscr S$ for which \eqref{deltad} holds 
with $b\in (0,1)$. 
 Let $N>0$,   $\gamma/N<p, q<\infty$ and $w \in A_{pN/\gamma}$. 
Let  $\varphi^{(\ell)}, \eta^{(\ell)} $ satisfy that 
$(\eta^{(m)}, X_k\varphi^{(\ell)})
\in \mathscr C_{N+\epsilon, N}^{(1)}$ for some $\epsilon>0$,  
$1\leq k \leq n$,  $1\leq \ell, m \leq M$.  
Let $\psi\in \mathscr S$ with $\int \psi \, dx=0$.  
If $(\eta^{(\ell)}, \psi) \in 
\mathscr C_{N+\epsilon, N}^{(1)}$  for $1\leq \ell\leq M$ 
with some $\epsilon>0$, then for $f\in \mathscr S$ we have 
\begin{equation*}\label{e4-6++}  
 \left\|\left(\int_0^\infty\left((f*\psi_t)_{N,t^{-1}}^{**}\right)^q
\, \frac{dt}{t}\right)^{1/q}\right\|_{L^p_w} 
\leq C\sum_{\ell=1}^M \left\|\left(\int_0^\infty|f*\varphi^{(\ell)}_t|^q\, 
\frac{dt}{t}\right)^{1/q}\right\|_{L^p_w}    
\end{equation*}   
for some positive constant $C$ independent of $f$.    
\end{theorem}  

\par 

For $q>0$ and $b\in (0,1)$, let 
\begin{equation*} 
\Delta^{(q)}_{\varphi, b}(f)(x)=\left(\sum_{j=-\infty}^\infty 
|f*\varphi_{b^{j}}(x)|^q\right)^{1/q}. 
\end{equation*}  
Put $\Delta_{\varphi, b}(f)=\Delta^{(2)}_{\varphi, b}(f)$.  
Then we can regard $\Delta_{\varphi, b}(f)$  as a discrete parameter analogue 
of $g_\varphi(f)$.   
We have discrete parameter versions of Theorems \ref{TT} and 
\ref{TT3-7} as follows. 
\begin{theorem}\label{Tdis} 
Let $0<p\leq 1$. There exists $d \in \Delta$ such that 
if functions 
$\varphi^{(\ell)}\in \mathscr{S}$, $1\leq \ell\leq M$, satisfy the 
conditions $(1)$ and $(2)$ of Theorem $\ref{TT}$,  
then we have 
$$c_p\|f\|_{H^p}\leq \sum_{\ell=1}^M\|\Delta_{\varphi^{(\ell)}, b}(f)\|_p 
\leq C_p\|f\|_{H^p} \quad \text{for $f\in H^p$.}   $$ 
\end{theorem}

\begin{theorem}\label{T3-7dis}    
 Let $N>0$,   $\gamma/N<p, q<\infty$ and $w \in A_{pN/\gamma}$. 
Let $\varphi^{(\ell)} \in \mathscr S$, $\int \varphi^{(\ell)}\, dx=0$, 
 $1\leq \ell\leq M$.  
Suppose that \eqref{deltad} holds with some $\eta^{(\ell)}\in \mathscr {S}$, 
$1\leq \ell \leq M$.  We also assume that 
$(\eta^{(m)}, X_k\varphi^{(\ell)})
\in \mathscr C_{N+\epsilon, N}^{(1)}$ for some $\epsilon>0$,  
$1\leq k \leq n$,  $1\leq \ell, m \leq M$.  
Let $\psi\in \mathscr S$ and $\int \psi \, dx=0$. 
If $(\eta^{(\ell)}, \psi) \in 
\mathscr C_{N+\epsilon, N}^{(1)}$  for $1\leq \ell\leq M$ 
with some $\epsilon>0$, then  we have 
$$ \left\|\left(\sum_{j=-\infty}^\infty 
\left((f*\psi_{b^{j}})_{N,b^{-j}}^{**}\right)^q\right)^{1/q}\right\|_{L^p_w} 
\leq 
C\sum_{\ell=1}^M \left\|\Delta^{(q)}_{\varphi^{(\ell)}, b}(f)\right\|_{L^p_w}  
 $$  
 for $f\in \mathscr S$ with a positive constant $C$ independent of $f$.  
\end{theorem}  
\par 
We can prove Theorems \ref{TT} and \ref{TT3-7} similarly to 
Theorems \ref{T} and \ref{T3-7}, respectively.  
Let 
\begin{equation*}\label{e6.2+}
\zeta= \sum_{j=-\infty}^0  \sum_{\ell=1}^M 
\varphi^{(\ell)}_{b^{j}}*\eta^{(\ell)}_{b^{j}},   
\end{equation*} 
where $\varphi^{(\ell)}$, $\eta^{(\ell)}$ are as in Theorem \ref{TT}. To 
prove Theorem \ref{TT} similarly to Theorem \ref{T}, 
it will be useful to note that $\zeta \in \mathscr S$. 
Also, methods which prove Theorems \ref{TT} and \ref{TT3-7} 
can be applied to show Theorems \ref{Tdis} and \ref{T3-7dis}, 
respectively. 
The restriction $q\geq 1$ is not needed in Theorems \ref{TT3-7} and 
\ref{T3-7dis}, which is assumed in Theorem \ref{T3-7}, since  estimates like  
\eqref{e4-5+} are not needed in the situation under the 
non-degeneracy condition \eqref{deltad}.   We can find relevant arguments 
in \cite{Sav, Sap}. 
We omit the details for the proofs of the results stated in this section.

\end{document}